\documentclass{amsart}
\usepackage{enumerate}
\usepackage{amsmath,amsfonts,amssymb,amsthm}

 \newtheorem{thm}{Theorem}[section]

 \theoremstyle{definition}
 \newtheorem{defn}[thm]{Definition}
 \theoremstyle{remark}
 \newtheorem{rem}[thm]{Remark}
 \newtheorem*{ex}{Example}
 \numberwithin{equation}{section}

\newcommand{\PSI}{\Psi}

\newcommand{\cS}{\mathbf{S}}
\newcommand{\ETA}[1]{\bigl(1+#1\bigr)^{-1}}
\newcommand{\EQU}{\mathcal{E}}

\newcommand{\al}{\alpha}
\newcommand{\e}{\varepsilon}
\newcommand{\de}{\delta}
\newcommand{\De}{\Delta}

\newcommand{\N}{\mathbb{N}}
\newcommand{\Z}{\mathbb{Z}}
\newcommand{\R}{\mathbb{R}}

\newcommand{\set}[1]{\{#1\}}
\newcommand{\bigprn}[1]{\bigl(#1\bigr)}
\newcommand{\Bigprn}[1]{\Bigl(#1\Bigr)}

\begin{document}

\title[A fixed point theorem for nonlinear contractions]
 {A Suzuki-type fixed point theorem\\ for nonlinear contractions}

\author[M. Abtahi]{Mortaza Abtahi}

\address{School of Mathematics and Computer Sciences\\
Damghan University, Damghan, \\
P.O.BOX 36715-364, Iran}

\email{abtahi@du.ac.ir}



\date{\today}

\begin{abstract}
  We introduce the notion of admissible functions and show
  that the family of L-functions introduced by Lim in
  [Nonlinear Anal. 46(2001), 113--120] and the family
  of test functions introduced by Geraghty in
  [Proc. Amer. Math. Soc., 40(1973), 604--608] are
  admissible. Then we prove that if $\phi$ is an admissible
  function, $(X,d)$ is a complete metric space, and $T$ is a mapping
  on $X$ such that, for $\al(s)=\phi(s)/s$, the condition
  $\ETA{\al(d(x,Tx))} d(x,Tx) < d(x,y)$ implies
  $d(Tx,Ty) < \phi(d(x,y))$, for all $x,y\in X$, then
  $T$ has a unique fixed point. We also show that our fixed point
  theorem characterizes the metric completeness of $X$.
\end{abstract}

\maketitle

\section{Introduction}

Throughout this paper, we denote by $\N$ the set of positive integers,
by $\Z^+$ the set of nonnegative integers, and
by $\R^+$ the set of nonnegative real numbers.
Given a set $X$ and a mapping $T:X\to X$, the $n$th iterate of
$T$ is denoted by $T^n$ so that $T^2x=T(Tx)$, $T^3x=T(T^2x)$ and so on.
A point $x\in X$ is called a \emph{fixed point} of $T$ if $Tx=x$.

Let $(X,d\,)$ be a metric space. A mapping $T:X\to X$ is called  a \emph{contraction}
if there is $r\in[0,1)$ such that $d(Tx,Ty)\leq rd(x,y)$, for all $x,y\in X$.
The following famous theorem is referred to as the Banach contraction
principle.

\begin{thm}[Banach \cite{Banach}]
\label{thm:Banach}
  If $(X, d\,)$ is a complete metric space, then every contraction $T$ on $X$
  has a unique fixed point.
\end{thm}

The Banach fixed point theorem is very simple and powerful. It became
a classical tool in nonlinear analysis with many generalizations;
see
\cite{Boyd-Wong,
Caristi-1976,
Ciric-1974,
Ekeland-1974,
Kirk-2003,
Meir-Keeler,
Nadler-1969,
Subrahmanyam-1974,
Suzuki-2001,
Suzuki-2004,
Suzuki-2005,
Suzuki-2008,
Suzuki-NA-2009}.
For instance, the following result due to Boyd and Wong \cite{Boyd-Wong}
is a great generalization of Theorem \ref{thm:Banach}.

\begin{thm}[{Boyd and Wong \cite{Boyd-Wong}}]
\label{thm:Boyd-Wong}
  Let $(X,d\,)$ be a complete metric space, and let $T$
  be a mapping on $X$. Assume there exists a function $\phi:\R^+\to\R^+$
  which is upper semi-continuous from the right, $\phi(s)<s$
  for $s>0$, and
  \begin{equation}\label{eqn:Boyd-Wong}
    \forall\, x,y\in X,\quad d(Tx,Ty) \leq \phi(d(x,y)).
  \end{equation}
  Then $T$ has a unique fixed point.
\end{thm}

Another interesting generalization of Banach contraction principle was
given by Meir and Keeler \cite{Meir-Keeler} as follows.

\begin{thm}[Meir and Keeler \cite{Meir-Keeler}]
  Let $(X,d)$ be a complete metric space and let $T$ be a Meir-Keeler contraction
  on $X$, i.e., for every $\e>0$, there exists $\de>0$ such that
  \begin{equation}
    \forall\, x,y\in X\
    (\e\leq d(x,y) < \e+\de\ \Longrightarrow\
    d(Tx,Ty)<\e).
  \end{equation}
  Then $T$ has a unique fixed point.
\end{thm}

Lim \cite{Lim} introduced the notion of L-functions and proved a characterization
of Meir-Keeler contractions that shows how much more general is Meir-Keeler's
result than Boyd-Wong's.
A function $\phi:[0,\infty)\to[0,\infty)$ is called an \emph{L-function} if
$\phi(0)=0$, $\phi(s)>0$ for $s>0$, and for every $s>0$ there exists
$\de>0$ such that $\phi(t)\leq s$ for all $t\in [s, s + \de]$.

\begin{thm}[Lim \cite{Lim}, see also \cite{Suzuki-MK-Lim}]
  Let $(X, d)$ be a metric space and let $T$ be a mapping on $X$.
  Then $T$ is a Meir-Keeler contraction if and only if there exists an
  L-function $\phi$ such that
  \[
    \forall\, x,y\in X,\quad d(Tx,Ty) < \phi(d(x,y)).
  \]
\end{thm}

There is an example of an incomplete metric space $X$ on which
every contraction has a fixed point, \cite{Connell-1959}.
This means that Theorem \ref{thm:Banach}
cannot characterize the metric completeness of $X$.
Recently, Suzuki in \cite{Suzuki-2008} proved the following
remarkable generalization of the classical Banach contraction principle
that characterizes the metric completeness of $X$.

\begin{thm}[Suzuki \cite{Suzuki-2008}]
\label{thm:Suzuki-2008}
  Define a function $\theta:[0, 1) \to (1/2, 1]$ by
  \begin{equation}\label{eqn:theta}
    \theta(r)=
      \left\{
        \begin{array}{ll}
          1, & \hbox{if $0\leq r \leq (\sqrt5-1)/2$;} \\
          (1-r)r^{-2}, & \hbox{if $(\sqrt5-1)/2 \leq r \leq 1/\sqrt2$;} \\
          (1+r)^{-1}, & \hbox{if $1/\sqrt2 \leq r <1$.}
        \end{array}
      \right.
  \end{equation}
  Let $(X,d\,)$ be a metric space. Then $X$ is complete if and only if
  every mapping $T$ on $X$ satisfying the following has a fixed point:
    \begin{itemize}
      \item  There exists $r\in[0, 1)$ such that
             \begin{equation}\label{eqn:suzuki-condition}
               \forall\, x,y\in X\
               \bigprn{\theta(r) d(x,Tx) \leq d(x,y) \ \Longrightarrow \
               d(Tx,Ty) \leq r d(x,y)}.
             \end{equation}
    \end{itemize}
\end{thm}

The above Suzuki's generalized version of Theorem \ref{thm:Banach}
initiated a lot of work in this direction and
led to some important contribution in metric fixed point
theory. Several authors obtained
variations and refinements of Suzuki's result; see
\cite{Enjouji-Nakanishi-Suzuki-2009,
Kikkawa-Suzuki-2008,
Kikkawa-Suzuki-NA-2008,
Popescu-2009,
Singh-Mishra-2010,
Singh-Pathak-Mishra-2010}.

A mapping $T$ on a metric space $X$ is called \emph{contractive} if
$d(Tx,Ty)<d(x,y)$, for all $x,y\in X$ with $x\neq y$.
Edelstein in \cite{Edelstein-62} proved that, on compact
spaces, every contractive mapping possesses a unique fixed
point theorem. Then in \cite{Suzuki-NA-2009} Suzuki generalized Edelstein's
result as follows. 

\begin{thm}[Suzuki \cite{Suzuki-NA-2009}]
\label{thm:Suzuki-2009-NA}
  Let $(X,d\,)$ be a compact metric space and let $T$ be a mapping on X.
  Assume that
  \begin{equation}\label{eqn:suzuki-NA-2009}
    \forall\, x,y\in X\
    \Bigprn{\frac12 d(x,Tx) < d(x,y)\ \Longrightarrow\
    d(Tx,Ty)<d(x,y)}.
  \end{equation}
  Then $T$ has a unique fixed point.
\end{thm}

\noindent
It is interesting to note that,
although the above Suzuki's theorem generalizes Edelstein's theorem
in \cite{Edelstein-62}, these two theorems are not of the same type
\cite{Suzuki-NA-2009}.

Recently, the author proved the following fixed point theorem
for contractive mapping
which is a Susuki-type generalization of \cite[Theorem 1.1]{Geraghty-73}
and characterizes metric completeness.

\begin{thm}[Abtahi \cite{Abtahi}]
\label{thm:Abtahi-Geraghty}
  A metric space $(X,d)$ is complete if and only if
  every mapping $T:X\to X$ satisfying the following two conditions
  has a fixed point;
  \begin{enumerate}[\upshape(i)]
    \item\label{item:1-main}
      $(1/2)d(x,Tx) < d(x,y)$ implies $d(Tx,Ty) < d(x,y)$, for all $x,y\in X$.

    \item\label{item:2-main}
    There exists a point $x\in X$ such that, for any two subsequences $\set{x_{p_n}}$ and
    $\set{x_{q_n}}$ of the iterations $x_n=T^nx$, $n\in\N$, if
    $d(x_{p_n},Tx_{p_n}) \leq d(x_{p_n},x_{q_n})$ for all $n$,
    and $\De_n\to1$, then $\de_n\to0$, where
    \[
    \de_n=d(x_{p_n},x_{q_n}), \quad
    \De_n=d(Tx_{p_n},Tx_{q_n})/\de_n.
    \]
  \end{enumerate}
\end{thm}

\begin{rem}
  In part \eqref{item:1-main} of the above theorem, $1/2$ is the best constant.
\end{rem}

\section{Existence of fixed points for nonlinear contractions}

\begin{defn}
  Let $\phi:\R^+\to\R^+$ be a function. Given a metric space
  $(X,d)$, a mapping $T:X\to X$ is called a
  \emph{generalized $\phi$-contraction} if
  \begin{equation}\label{eqn:phi-contraction}
    \forall\, x,y\in X\
    \Bigprn{x\neq y,\ d(x,Tx) \leq d(x,y)\ \Longrightarrow\
    d(Tx,Ty) < \phi(d(x,y))}.
  \end{equation}
  We call $\phi$ \emph{admissible} if, for every metric space $X$,
  for every generalized $\phi$-contraction $T$
  on $X$, and for every choice of initial point $x\in X$, the iterations
  $x_n=T^n x$, $n\in\N$, form a Cauchy sequence.

\end{defn}

\begin{thm}\label{thm:L-function is admissible}
  Every L-function is admissible.
\end{thm}

\begin{proof}
  Let $\phi$ be an L-function and let $T$ be a generalized $\phi$-contraction
  on a metric space $X$. Fix $x\in X$ and let $x_n=T^nx$, $n\in\N$. If $d(x_m,x_{m+1})=0$,
  for some $m$, then $x_n=x_m$ for $n\geq m$ and there is nothing
  to prove. Assume that $d(x_n,x_{n+1})>0$ for all $n$. Since
  $d(x_n,Tx_n) \leq d(x_n,Tx_n)$ and $x_n\neq x_{n+1}$, condition \eqref{eqn:phi-contraction}
  implies that, for every $n\in\N$,
  \[
    d(x_{n+1},x_{n+2}) < \phi(d(x_{n},x_{n+1}))
    \leq d(x_{n},x_{n+1}).
  \]

  \noindent
  This shows that the sequence $\set{d(x_n,x_{n+1})}$ is strictly decreasing
  and thus it converges to some point $s\geq0$. If $s>0$, since $\phi$ is an L-function,
  there is $\de>0$ such that $\phi(t)\leq s$ for $s \leq t \leq s+\de$. Take $n\in\N$
  large enough so that $s \leq d(x_n,x_{n+1}) \leq s+\de$. Then
  \[
    d(x_{n+1},x_{n+2}) < \phi(d(x_{n},x_{n+1})) \leq s,
  \]
  which is a contradiction. Hence $d(x_n,x_{n+1})\to0$.

  Next, we show that $\set{x_n}$ is a Cauchy sequence. To this
  end we adopt the same method used by Suzuki in \cite{Suzuki-MK-Lim}.
  Fix $\e>0$ and let $s=\e/2$. Since $\phi$ is an L-function,
  there exists $\de\in(0,s)$ such that
  $\phi(t)\leq s$ for $s \leq t \leq s+\de$. Since $d(x_n,x_{n+1})\to0$,
  there is $N\in\N$ such that $d(x_n,x_{n+1})<\de$ for $n\geq N$.
  We show that
  \begin{equation}\label{eqn:claim}
    d(x_n,x_{n+m})<\de+s\leq \e, \qquad (n\geq N,\,m\in\N).
  \end{equation}

  \noindent
  For every $n\geq N$, we prove \eqref{eqn:claim} by induction on $m$.
  It is obvious that \eqref{eqn:claim} holds for $m=1$. Assume that
  \eqref{eqn:claim} holds for some $m\in\N$. Then
  $\phi(d(x_n,x_{n+m}))\leq s$. Now, if
  $d(x_n,Tx_n)\leq d(x_n,x_{n+m})$ then \eqref{eqn:phi-contraction}
  shows that $d(x_{n+1},x_{n+m+1})<\phi(d(x_n,x_{n+m}))$ and thus
  \[
    d(x_n,,x_{n+m+1})\leq d(x_n,x_{n+1})+d(x_{n+1},x_{n+m+1})<\de+s\leq\e.
  \]

  \noindent
  If $d(x_n,x_{n+m})<d(x_n,Tx_n)$ then $d(x_n,x_{n+m})<\de$ and thus
  \[
    d(x_n,,x_{n+m+1})\leq d(x_n,x_{n+m})+d(x_{n+m},x_{n+m+1})<\de+\de \leq \de+s \leq\e.
  \]

  \noindent
  Therefore \eqref{eqn:claim} is verified and $\set{x_n}$ is a Cauchy sequence.
\end{proof}

As in \cite{Geraghty-73}, we define $\cS$ to be the class of all functions
$\al:\R^+\to [0,1]$ such that, for any sequence $\set{s_n}$ of
positive numbers, if $\al(s_n)\to1$ then $s_n\to0$.

\begin{thm}\label{thm:sl(s)s is admissible}
  If $\al\in \cS$, the function $\phi(s)=\al(s)s$ is admissible.
\end{thm}

\begin{proof}
  Let $\al\in\cS$ and define $\phi(s)=\al(s)s$.
  Let $T$ be a generalized $\phi$-contraction on a metric space $X$, let
  $x\in X$ and let $x_n=T^nx$, $n\in\N$. Let $s_n=d(x_n,x_{n+1})$.
  As in the proof of Theorem \ref{thm:L-function is admissible}, we assume
  that $s_n>0$ for all $n$. Then $s_{n+1}<\al(s_n)s_n$ and thus $s_n\to s$
  for some point $s\geq0$. If $s>0$ then $s_{n+1}/s_n\to1$ and thus
  $\al(s_n)\to1$. Since $\al\in\cS$, we must have $s=0$ which is
  a contradiction. Hence $s=0$ and $d(x_n,x_{n+1})\to0$.

  For every $n\in\N$, choose $k_n\in\N$ such that
  $d(x_m,x_{m+1})<1/n$ for $m\geq k_n$. If $\set{x_n}$ is not a Cauchy sequence,
  there exist $\e>0$ and sequences $\set{p_n}$ and $\set{q_n}$ of positive integers
  such that $q_n>p_n\geq k_n$ and $d(x_{p_n},x_{q_n})\geq\e$. We also
  assume that $q_n$ is the least such integer so that
  $d(x_{p_n},x_{q_n-1})<\e$. Therefore,
  \[
    \e \leq d(x_{p_n},x_{q_n})
       \leq d(x_{p_n},x_{q_n-1})+d(x_{q_n-1},x_{q_n}) < \e+1/n.
  \]

  \noindent
  This shows that $s_n\to\e$. Since we have, for every  $n\in\N$,
  \[
    d(x_{p_n},Tx_{p_n}) \leq d(x_{p_n},x_{q_n})
     < d(x_{p_n},x_{q_n}),
  \]
  condition \eqref{eqn:phi-contraction} shows that
  $d(x_{p_n+1},x_{q_n+1})<\al(s_n)s_n$. Hence we have
  \begin{align*}
    s_n = d(x_{p_n},x_{q_n})
       & \leq d(x_{p_n},x_{p_n+1}) + d(x_{p_n+1},x_{q_n+1}) + d(x_{q_n+1},x_{q_n}) \\
       & < 2/n + \al(s_n)s_n.
  \end{align*}

  \noindent
  Dividing the above inequality by $s_n$, since $\al(s_n)\leq 1$,
  we get $\al(s_n)\to1$ and thus
  $s_n\to0$ which is a contradiction. Therefore, $\set{x_n}$ is a Cauchy sequence.
\end{proof}

\begin{defn}\label{defn:class PSI and PHI}
  A function $\al:\R^+\to(0,1]$ is said to be of class $\PSI$,
  written $\al\in\PSI$, if the function $\phi(s)=\al(s)s$ is
  admissible and, moreover, there exists $\de>0$ such that
  \begin{equation}\label{eqn:condition-on-alpha}
    0<t<\de,\ 0 < s < \al(t)t \ \Longrightarrow\  \al(t) \leq \al(s).
  \end{equation}

  \noindent
  Given two points $x$ and $y$ in a metric space $(X,d\,)$,
  by $\al(x,y)$ we always mean $\al(d(x,y))$.
\end{defn}

\begin{ex}
  Every decreasing function $\al:\R^+\to(0,1]$ is of class $\PSI$. For example,
  if $\al(s)=(1+s)^{-1}$, then $\al\in\PSI$.
\end{ex}

\begin{thm}\label{thm:main}
  Let $(X,d\,)$ be a complete metric space and let $T$ be a mapping on $X$.
  Assume that there is a function $\al\in\PSI$ such that
  \begin{equation}\label{eqn:our-condition-on-T}
    \ETA{\al(x,Tx)}d(x,Tx) < d(x,y) \ \Longrightarrow \
    d(Tx,Ty) < \al(x,y)d(x,y),
  \end{equation}
  holds for every $x,y\in X$.
  Then $T$ has a unique fixed point.
\end{thm}

\begin{proof}
 First, let us prove the uniqueness part of the theorem.
 If $z\in X$ is a fixed point of $T$ and $y\neq z$ then
 \[
   \ETA{\al(z,Tz)}d(z,Tz) < d(z,y),
 \]
 and thus by \eqref{eqn:our-condition-on-T} we have
 $d(Tz,Ty) < d(z,y)$. Since $Tz=z$, we must have $Ty\neq y$,
 i.e., $y$ is not a fixed point of $T$.

 Now, we prove the existence of the fixed point.
 Take two points $x,y\in X$ with $x\neq y$. If $d(x,Tx)\leq d(x,y)$
 then $\ETA{\al(x,Tx)}d(x,Tx)<d(x,y)$, because $\al(x,Tx)>0$
 and $d(x,y)>0$.
 Hence $T$ satisfies condition \eqref{eqn:phi-contraction} with
 $\phi(s)=\al(s)s$. Fix $x\in X$ and define $x_n=T^nx$,
 $n\in \N$. Since the function $\phi(s)=\al(s)s$
 is admissible, the sequence $\set{x_n}$ is Cauchy.
 Since $X$ is complete,
 there is $z\in X$ such that $x_n\to z$. Next, we show that $Tz=z$.

 If $x_m=Tx_m$ for some $m$, the $x_n=z$ for $n\geq m$ and $Tz=z$.
 We assume that $x_n\neq Tx_n$ for all $n$.
 Since $\al\in\PSI$, condition \eqref{eqn:condition-on-alpha} holds
 for some $\de>0$. Take a positive number $N$ such that
 $d(x_n,Tx_n)<\de$ for $n\geq N$. Then
 \[
   0 < d(Tx_n,T^2x_n) < \al(x_n,Tx_n)d(x_n,Tx_n),
   \qquad (n\geq N),
 \]

 \noindent
 and condition \eqref{eqn:condition-on-alpha} shows that
 $\al(x_n,Tx_n) \leq \al(Tx_n,T^2x_n)$, for $n\geq N$, so that
 \begin{equation}\label{eqn:AUX}
    \frac1{1+\al(x_n,Tx_n)}+\frac{\al(x_n,Tx_n)}{1+\al(Tx_n,T^2x_n)}
    \leq 1.
 \end{equation}

 \noindent
 We claim that
 \begin{equation}\label{eqn:either-or}
   \forall\,n\geq N, \quad
   \left\{
     \begin{array}{l}
       \ETA{\al(x_n,Tx_n)}d(x_n,Tx_n) < d(x_n,z),\\[1ex]
       \text{\quad or}\\[1ex]
       \ETA{\al(Tx_n,T^2x_n)}d(Tx_n,T^2x_n) < d(x_{n+1},z).
     \end{array}
   \right.
 \end{equation}

 \noindent
 If \eqref{eqn:either-or} fails to hold, then, for some $n\geq N$, we have
 \begin{align*}
       d(x_n,z) & \leq \ETA{\al(x_n,Tx_n)}d(x_n,Tx_n),\\
       d(x_{n+1},z) & \leq \ETA{\al(Tx_n,T^2x_n)}d(Tx_n,T^2x_n).
 \end{align*}

 \noindent
 Using \eqref{eqn:AUX}, we have
 \begin{align*}
   d(x_n & , Tx_n) \leq d(x_n,z)+d(Tx_n,z) \\
    & \leq \ETA{\al(x_n,Tx_n)}d(x_n,Tx_n) + \ETA{\al(Tx_n,T^2x_n)}d(Tx_n,T^2x_n) \\
    & < \bigl[\ETA{\al(x_n,Tx_n)}+\ETA{\al(Tx_n,T^2x_n)}\al(x_n,Tx_n)\bigr]d(x_n,Tx_n)\\
    & \leq d(x_n,Tx_n).
 \end{align*}
 This is absurd and thus \eqref{eqn:either-or} must hold.
 Now condition \eqref{eqn:our-condition-on-T} together
 with \eqref{eqn:either-or} imply that
 \begin{equation}\label{eqn:either-or-II}
 \begin{split}
   \forall n\geq N,\quad
   d(x_{n+1},Tz) < \phi(d(x_n,z))
   \ \text{or}\ d(x_{n+2},Tz) < \phi(d(x_{n+1},z)).
 \end{split}
 \end{equation}

 \noindent
 Since $x_n\to z$ and $\phi(s)\leq s$, condition \eqref{eqn:either-or-II} implies
 the existence of a subsequence of $\set{x_n}$ that converges
 to $Tz$. This shows that $Tz = z$.
\end{proof}

The following theorem states that, for a certain family of functions
$\al\in\PSI$, the coefficient $1/(1+\al)$ in Theorem \ref{thm:main}
is the best.

\begin{thm}\label{thm:best-constant}
 Let the function $\al\in\PSI$ satisfy the following condition;
 \begin{equation}
   \al_0=\liminf\limits_{s\to0+}\al(s)>1/\sqrt2.
 \end{equation}

 \noindent
 Then, for every constant $\eta$ with $\eta>1/(1+\al_0)$,
 there exist a complete metric space $(X, d\,)$ and a mapping
 $T:X\to X$ such that $T$ does not have a fixed point and
 \[
    \forall\, x, y\in X,\
    \bigprn{\eta d(x,Tx) < d(x,y) \ \Longrightarrow \  d(Tx,Ty) < \al(x,y)d(x,y)}.
 \]
\end{thm}

\begin{proof}
  Take a number $r\in(1/\sqrt2,\al_0)$ such that $(1+r)^{-1}<\eta$.
  The proof of Theorem~3 in \cite{Suzuki-2008} shows that there exist
  a closed and bounded subset $X$ of $\R$ and a mapping $T:X\to X$
  such that $T$ does not have a fixed point and
  \begin{equation}\label{eqn:from-Suzuki}
     \forall\, x, y\in X\
     \Bigprn{(1+r)^{-1}|x-Tx| < |x-y| \ \Longrightarrow \
     |Tx-Ty| \leq r|x-y|}.
  \end{equation}

  \noindent
  Since $r<\liminf\limits_{s\to0+}\al(s)$, there exists $\de>0$ such that
  $r<\al(s)$ for $s\in(0,\de)$. Since $X$ is bounded, there is a constant
  $M$ such that $|x-y| < M\de$, for all $x,y\in X$.
  Now, define a metric $d$ on $X$ by
  \[
    d(x,y) = \frac1M|x-y|, \qquad (x,y\in X).
  \]

  \noindent
  For $x,y\in X$, if $\eta d(x,Tx) < d(x,y)$ then
  $(1+r)^{-1}d(x,Tx) < d(x,y)$. Now, condition \eqref{eqn:from-Suzuki}
  and the fact that $d(x,y)<\de$ shows that
  \[
    d(Tx,Ty) \leq r d(x,y) < \al(d(x,y))d(x,y).
  \]
\end{proof}

\begin{ex}
  For the function $\al(s)=(1+s)^{-1}$, we have $\al_0=1$.
  Hence $\al$ satisfies the condition
  in Theorem \ref{thm:best-constant}.
\end{ex}

\section{Metric Completion}

In this section, we discuss the metric completeness. Let $X$ be a nonempty set.
We say that two metrics $d$ and $\rho$ on $X$ are equivalent if
they generate the same topology and the same Cauchy sequences.
Given a metric $\rho$ on $X$, we denote the family of all metrics $d$ on $X$
equivalent to $\rho$ by $\EQU_\rho$. It is obvious that $(X,\rho)$ is complete
if and only if $(X,d)$, for some $d\in\EQU_\rho$, is complete if and only if
$(X,d)$, for all $d\in\EQU_\rho$, is complete. For a function $\al\in\PSI$, we
define
\[
  \al_0=\liminf\limits_{s\to0+}\al(s),
\]

\noindent
and we denote by $\PSI^+$ the family of those functions $\al\in\PSI$
with $\al_0>0$.

\begin{thm}
  For a metric space $(X, \rho)$ the following are equivalent:
  \begin{enumerate}
    \item \label{item:X is complete}
    The space $(X,\rho)$ is complete.

    \item \label{item:for-all-al}
    For every $\al\in\PSI$ and $d\in\EQU_\rho$, every mapping
    $T$ satisfying \eqref{eqn:our-condition-on-T} has a fixed point.

    \item \label{item:for-some-al-eta}
    For some $\al\in\PSI^+$ and $\eta\in(0,1/2]$, and
    for all $d\in\EQU_\rho$, every mapping $T$ satisfying
    the following condition has a fixed point;
   \begin{equation}\label{eqn:our-condition-on-T-eta-al}
     \forall\, x, y\in X,\
     \bigprn{\eta d(x,Tx) < d(x,y) \ \Longrightarrow \  d(Tx,Ty) < \al(x,y)d(x,y)}.
   \end{equation}
\end{enumerate}
\end{thm}

\begin{proof}
  The implication $\eqref{item:X is complete}\Rightarrow\eqref{item:for-all-al}$
  follows from Theorem \ref{thm:main}. The implication
  $\eqref{item:for-all-al}\Rightarrow\eqref{item:for-some-al-eta}$ is clear
  because, for $\eta\leq 1/2$, condition \eqref{eqn:our-condition-on-T-eta-al}
  implies condition \eqref{eqn:our-condition-on-T}.

  To prove $\eqref{item:for-some-al-eta}\Rightarrow\eqref{item:X is complete}$,
  towards a contradiction, assume that the metric space $(X,\rho)$ is not complete.
  Take a number $r\in(0,\al_0)$ and let $\de$ be a positive number such that
  $r<\al(s)$ for all $s\in(0,\de)$. Define a metric $d$ on $X$ as follows:
  \[
    d(x,y)= \de \frac{\rho(x,y)}{1+\rho(x,y)},
    \qquad (x,y\in X).
  \]

  \noindent
  Then $d\in \EQU_\rho$ and thus $(X,d)$ is not complete.
  The proof of Theorem 4 in \cite{Suzuki-2008} shows that there exists
  a mapping $T:X\to X$ with no fixed point such that
  \[
   \forall\, x, y\in X,\
     \bigprn{\eta d(x,Tx) < d(x,y) \ \Longrightarrow \  d(Tx,Ty) \leq r d(x,y)}.
  \]

  \noindent
  Since $d(x,y)<\de$, we have $r d(x,y) < \al(x,y)d(x,y)$ and thus
  $T$ satisfies \eqref{eqn:our-condition-on-T-eta-al}. This is
  a contradiction.
\end{proof}


\end{document}